\RequirePackage[l2tabu,orthodox,abort]{nag}

\documentclass{birkjour}[draft]

\usepackage{fixltx2e}

\usepackage{amssymb,amsfonts,amsmath,
amsrefs,enumerate,color,url,
mathbbol,dsfont,upgreek}
\usepackage{mathrsfs}
\usepackage{tikz}

\usepackage[T1]{fontenc}
\usepackage[cp1250]{inputenc}

\usepackage[normalem]{ulem}

\usepackage[strict=true]{csquotes}

\usepackage{graphicx,tikz}

 \newtheorem{thm}{Theorem}[section]
 \newtheorem{cor}[thm]{Corollary}
 \newtheorem{lem}[thm]{Lemma}
 \newtheorem{prop}[thm]{Proposition}
 \theoremstyle{definition}
 
 \theoremstyle{remark}
 \newtheorem{rem}{Remark}
 
 \numberwithin{equation}{section}

\newcommand{\freg}{\mathsf F_{\textnormal {reg}}}
 \DeclareMathOperator{\Cos}{Cos}

\usepackage{mathtools}

\DeclareMathAlphabet{\mathpzc}{OT1}{pzc}{m}{it}

\usepackage{enumitem}
\renewcommand{\theenumi}{\textup{(\alph{enumi})}}
\renewcommand{\labelenumi}{\theenumi}

\usepackage{hyperref}
\begin{document}

\title[From snapping-out BM to Walsh's spider]{From snapping out Brownian motions to Walsh's spider processes on star-like graphs}

\author[A. Bobrowski]{Adam Bobrowski}

\address{
Lublin University of Technology\\
Nadbystrzycka 38A\\
20-618 Lublin, Poland}
\email{a.bobrowski@pollub.pl}

\author[E. Ratajczyk]{El\.zbieta Ratajczyk}
\address{
Lublin University of Technology\\
Nadbystrzycka 38A\\
20-618 Lublin, Poland}

\email{e.ratajczyk@pollub.pl}

\newcommand{\cxi}{(\xi_i)_{i\in \N} }
\newcommand{\lam}{\lambda}
\newcommand{\eps}{\varepsilon}
\newcommand{\ud}{\, \mathrm{d}}
\newcommand{\mud}{\mathrm{d}}
\newcommand{\pr}{\mathbb{P}}
\newcommand{\s}{\mathcal{S}}
\newcommand{\h}{\mathcal{H}}
\newcommand{\ai}{\mathcal{I}}
\newcommand{\R}{\mathbb{R}}
\newcommand{\C}{\mathbb{C}}
\newcommand{\Z}{\mathbb{Z}}
\newcommand{\N}{\mathbb{N}}
\newcommand{\Y}{\mathbb{Y}}
\newcommand{\e}{\mathrm {e}}
\newcommand{\tif}{\widetilde {f}}
\newcommand{\Id}{{\mathrm{Id}}}
\newcommand{\cic}{C_{\mathrm{mp}}}
\newcommand{\pol}{{\textstyle \frac 12}}
\newcommand{\es}{\textnormal T}

\newcommand{\cod}{\mathfrak C_{\mathrm{odd}}[-\infty,\infty]}
\newcommand{\cev}{\mathfrak C_{\mathrm{even}}[-\infty,\infty]}
\newcommand{\cevr}{C_{\mathrm{even}}(\mathbb{R})}
\newcommand{\codr}{C_{\mathrm{odd}}(\mathbb{R})}
\newcommand{\cez}{C_0(0,1]}
\newcommand{\fod}{f_{\mathrm{odd}}} 
\newcommand{\fev}{f_{\mathrm{even}}} 
\newcommand{\sem}[1]{\mbox{$\left \{\e^{t{#1}}, {t \ge 0}\right \}$}}
\newcommand{\semi}[1]{\mbox{$\left ({#1}\right )_{t > 0}$}}
\newcommand{\semt}[2]{\mbox{$\left (\e^{t{#1}} \otimes_\varepsilon \e^{t{#2}} \right )_{t \ge 0}$}}
\newcommand{\tr}{\textcolor{red}}
\newcommand{\wt}{\widetilde}

\newcommand{\tcm}{\textcolor{magenta}}
\newcommand{\ecm}{\textcolor{olive}}
\newcommand{\tcb}{\textcolor{blue}}
\newcommand{\dx}{\ \textrm {d} x}
\newcommand{\dy}{\ \textrm {d} y}
\newcommand{\dz}{\ \textrm {d} z}
\newcommand{\di}{\textrm{d}}
\newcommand{\tcg}{\textcolor{green}}
\newcommand{\lc}{\mathfrak L_c}
\newcommand{\ls}{\mathfrak L_s}
\newcommand{\grat}{\lim_{t\to \infty}}
\newcommand{\grar}{\lim_{r\to 1-}}
\newcommand{\graR}{\lim_{R\to 1+}}
\newcommand{\grak}{\lim_{\kappa \to \infty}}
\newcommand{\gra}{\lim_{x\to \infty}}
\newcommand{\grae}{\lim_{\eps \to 0}}
\newcommand{\gran}{\lim_{n\to \infty}}
\newcommand{\rez}[1]{\left (\lam - #1\right)^{-1}}
\newcommand{\papa}{\hfill $\square$}
\newcommand{\papap}{\end{proof}}
\newcommand {\x}{\cerf}
\newcommand{\aex}{A_{\mathrm ex}}
\newcommand{\jcg}[1]{\left ( #1 \right )_{n\ge 1} }
\newcommand{\injtp}{\x \hat \otimes_{\varepsilon} \y}
\newcommand{\pin}{\|_{\varepsilon}}
\newcommand{\mc}{\mathcal}
\newcommand{\inter}{\left [0, 1\right ]}
\newcommand{\ha}{\mathfrak {H}}
\newcommand{\dom}[1]{D(#1)}
\newcommand{\mquad}[1]{\quad\text{#1}\quad}
\newcommand{\lil}{\lim_{\lam \to \infty}}
\newcommand{\lilz}{\lim_{\lam \to 0}}

\newcommand{\al}{\alpha}
\newcommand{\ce}{\mathcal C}
\newcommand{\cerl}{C[-\infty,0]}

\newcommand{\cerp}{C[0,\infty]}
\newcommand{\cer}{C[-\infty,\infty]}
\newcommand{\cerf}{\mathfrak  C(\R_\sharp)}
\newcommand {\xo}{\mathfrak  C_{\textnormal{ov}}(\R_\sharp)}
\newcommand{\cerr}{\mathfrak C_R^\alpha}
\newcommand{\cef}{\mathfrak C_F^\alpha}
\newcommand{\fo}{f_{\textrm{o}}}
\newcommand{\fe}{f_{\textrm{e}}}
\newcommand{\wh}{\widehat}

\newcommand{\rla}{R_\lam}
\newcommand{\ced}{C_{\textnormal{D}}}

\makeatletter
\newcommand{\normt}{\@ifstar\@normts\@normt}
\newcommand{\@normts}[1]{%
  \left|\mkern-1.5mu\left|\mkern-1.5mu\left|
   #1
  \right|\mkern-1.5mu\right|\mkern-1.5mu\right|
}
\newcommand{\@normt}[2][]{%
  \mathopen{#1|\mkern-1.5mu#1|\mkern-1.5mu#1|}
  #2
  \mathclose{#1|\mkern-1.5mu#1|\mkern-1.5mu#1|}
}
\makeatother

\thanks{Version of \today}

\newcommand{\Gs}{\mathfrak G_\textnormal{s}}
\newcommand{\G}{\mathfrak G_\ka}
\newcommand{\F}{{\mathfrak F_{\ka}}}
\newcommand{\cfk}{(f_i)_{i\in \mc K}}
\newcommand{\wcfk}{ ( \widetilde{f}_i)_{i\in \mc K}}
\newcommand{\cgk}{\left ( g_i \right )_{i\in \mc K}}
\newcommand{\sumi}{\sum_{i\in \mc K}}
\newcommand{\sumj}{\sum_{j\in \mc K}}
\newcommand{\sumjni}{\sum_{j\not =i}}
\newcommand{\slam}{\sqrt{\lam}}

\newcommand{\ka}{k}
\newcommand{\aso}{A^{\textnormal{s-o}}_\mathpzc p}
\newcommand{\asoc}{A^{\textnormal{s-o}}_\mathpzc c}
\newcommand{\asoce}{A^{\textnormal{s-o}}_{\mathpzc c(\eps)}}
\newcommand{\asoe}{A^{\textnormal{s-o}}_{\mathpzc p(\eps)}}
\newcommand{\aws}{A^{\textnormal{sp}}_{\mathpzc q}}
\newcommand{\awsa}{A^{\textnormal{sp}}_{\alpha(\mathpzc c)}}

\newcommand{\awsp}{A^{\textnormal{sp}}_{\mathpzc {q(p)}}}

\newcommand{\kak}{K_{1,\ka}}
\newcommand{\Fze}{\mathfrak F_\ka^0}

\newcommand{\rezasoe}{(\lam -\asoe)^{-1}}
\newcommand{\semasoe}{\{\e^{t\asoe},t \ge 0\}}
\newcommand{\rezw}{(\lam -\awsp)^{-1}}

\begin{abstract} By analyzing matrices involved, we prove that a snapping-out Brownian motion with large permeability coefficients is a good approximation of Walsh's spider process  on the star-like graph $K_{1,\ka}$. Thus, the latter process can be seen as a Brownian motion perturbed by a trace of semi-permeable membrane at the graph's center. \end{abstract}

\maketitle

\section{Introduction}\label{intro}

\subsection{Sticky snapping-out Brownian motion}
\label{sec:sso}
Given a natural number $\ka \ge 2$ and non-negative parameters 
\[ a_i,b_i, c_i, \qquad i  \in \mc K \coloneqq \{1,\dots, \ka\}\]
with $b_i>0$, we think of the following Markov process on the compact space
\[ S_\ka \coloneqq \bigcup _{i \in \mc K}( \{i\}\times [0,\infty]). \]
While on the $i$th copy of $[0,\infty]$, our process is initially indistinguishable from the \emph{sticky Brownian motion} with stickiness coefficient $a_i/b_i$, as described, for example, in \cite{liggett} (see also \cite{newfromold}); in the particular case of $a_i=0$, the sticky Brownian motion is the reflected Brownian motion. However, when the time spent at the $i$th copy of $0$ exceeds an exponential time (independent of the Brownian motion) with parameter $c_i$, the process jumps to one of the points $(j,0), j \not = i$, all choices being equally likely. At the moment of jump the process forgets its past and starts to behave like a sticky Brownian motion on the $j$th copy of $[0,\infty]$, and so on (see Figure \ref{rys1}).

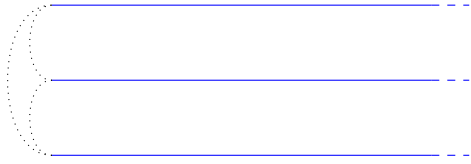
\begin{figure}

\begin{tikzpicture}
      
  \foreach \n in {1,...,3}{
  \draw [blue,-] (0,\n) -- (5,\n);  \draw [blue,dashed] (5,\n) -- (5.5,\n); 
}  \draw [black,dotted] (0,1) to [out=180,in=180] (0,2);   \draw [black,dotted] (0,2) to [out=180,in=180]  (0,3); \draw [black,dotted] (0,1) to [out=180,in=180] (0,3);   
      \end{tikzpicture}
\caption{\footnotesize{Sticky snapping out Brownian motion  is a Feller process on  $\ka$ copies of $[0,\infty]$ (here $\ka =3$), which on the $i$th copy behaves like a one-dimensional sticky Brownian motion with stickiness coefficient $a_i/b_i$.  After spending enough time at $(0,i)$ the process jumps to one of the points $(0,j), j\not =i$ to continue its motion on the corresponding copy of $[0,\infty]$, and so on. Times between jumps are governed by parameters $c_i$.}}\label{rys1} 
\end{figure}

The so-defined process has Feller property and thus can be characterized by means of a Feller generator in $C(S_\ka)$, the space of continuous functions on $S_\ka$.  To describe this generator, say, $\aso$ (subscript $\mathpzc p $ is a shorthand for the ordered set of parameters $a_i,b_i,c_i$ described above), in detail, we introduce first the space $\cerp$ of real-valued, continuous functions $f$ on $[0,\infty)$ such that the limits $\gra f(x) $ exist and are finite, and note that  $C(S_\ka)$ is isometrically isomorphic to  
\[ \F \coloneqq (\cerp )^\ka , \]
the Cartesian product of $\ka$ copies of $\cerp$; we  identify a sequence $\cfk$ with the function $f: S_\ka \to \R$ given by $f(i,x) = f_i(x), i\in \mc K, x\in [0,\infty)$. Then, we define the domain $\dom{\aso}$  of $\aso$ as composed of $\cfk\in \F$ such that 
\begin{enumerate} 
\item  each $f_i$ is twice continuously differentiable with $f_i'' \in \cerp $,
\item we have
\begin{equation}a_if_i''(0) - b_i f_i'(0) = c_i  \Big[  {{\textstyle \frac 1{\ka-1}} \sum\limits_{j \neq i}} f_j(0) - f_i(0) \Big ], \qquad i \in \mc K. \label{intro:1}\end{equation}
\end{enumerate}  
Moreover, for such $\cfk$, we agree that \( \aso \cfk = \left ( f_i'' \right )_{i \in \mc K} .\)

Boundary/transmission conditions of the type  \eqref{intro:1} have been studied and employed extensively by a number of authors in a variety of contexts (see e.g. the abundant bibliography in \cite{knigazcup}). Among more recent literature involving relatives of \eqref{intro:1}, one could mention the model  of inhibitory synaptic receptor dynamics of P.~Bressloff
\cite{bressloff2023}; see also \cite{bressloff2022}, and references given in these papers. It seems that the probabilistic meaning of relations \eqref{intro:1},  in the case of $\ka =2, a_1=a_2=0$,  has been first explained in \cite{bobmor}, where they were put in the context of Feller boundary conditions; the subsequent \cite{lejayn} provides a construction of the underlying process and introduces the name \emph{snapping-out} Brownian motion.

Our main theorem in this paper says that snapping-out Brownian motion with very large coefficients $c_i$ can be used as an approximation for the famous Walsh's spider process.

\subsection{Walsh's sticky process}\label{skew}
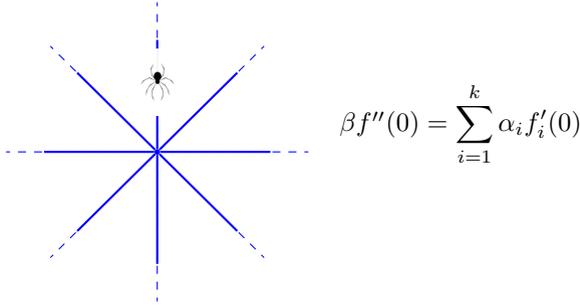
\begin{figure}
\begin{tikzpicture}
      \node[circle,fill=blue,inner sep=0pt,minimum size=2pt] at (360:0mm) (center) {};
    \foreach \n in {1,...,8}{
        \node [circle,fill=blue,inner sep=0pt,minimum size=0.1pt] at ({\n*360/8}:2cm) (n\n) {};
        \draw [blue,dashed](center)--(n\n);}
     \node[circle,fill=blue,inner sep=0pt,minimum size=2pt] at (360:0mm) (center) {};
    \foreach \n in {1,...,8}{
        \node [circle,fill=blue,inner sep=0pt,minimum size=0.1pt] at ({\n*360/8}:1.5cm) (n\n) {};
        \draw [blue,thick](center)--(n\n);}   
                 \node [above] at (4,-0.3) {$\beta f''(0)= \displaystyle\sum_{i=1}^\ka \alpha_i f_i'(0)$}; 
        \node [below] at (0,1.5) {\includegraphics[scale=0.165]{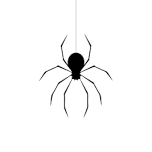}};\end{tikzpicture}
\caption{\footnotesize{The infinite star-like graph $\kak$  with $\ka=8$ edges. Walsh's sticky process on $\kak$ is a Feller process whose behavior at the graph's center is characterized by the boundary condition visible above; outside of the center, on each of the edges, the process behaves like a standard one-dimensional Brownian motion.}}\label{rys2} 
\end{figure}

Walsh's sticky process (a generalization of Walsh's spider process) is a Feller process on the space obtained from $S_\ka$ by lumping together all the points $(i,0)\in S_\ka, i \in \mc K$, that is, on the infinite metric graph $\kak$ depicted at Figure~\ref{rys2}. The space $C(\kak)$ of continuous functions on $\kak$ can be identified with the subspace \[ \Fze \coloneqq \{ \cfk \in \F\,|  f_i(0)=f_j(0), i,j \in \mc K\} \subset \F. \]
The generator, say, $\aws$, of Walsh's sticky process is characterized by non-negative numbers $\beta $ and $\alpha_i, i \in \mc K$ such that 
 \( \beta + \sumi   \alpha_i = 1 \) --- subscript $\mathpzc q$ denotes the ordered set of such numbers.  Namely, see  \cite{kostrykin2}*{Thm. 2.2},
we define the domain $\dom{\aws}\subset \Fze$ of $\aws$ as composed of $\cfk\in \Fze$ such that 
\begin{enumerate} 
\item each $f_i$ is twice continuously differentiable with $f_i'' \in \cerp $,
\item $f_i''(0)=f_j''(0), i,j\in \mc K$,
\item $\beta f''(0) = \sum\limits_{i \in \mc K}\alpha_i f_i'(0),$
\end{enumerate} and agree that 
\( \aws \cfk =  \left ( f_i'' \right )_{i \in \mc K} .\)
Notably, by condition (b) above, $\aws f$ belongs to $\Fze$ and in particular it is meaningful to speak of $f''(0)$ for $f\in \dom{\aws}$, even though $f'(0)$ is not well-defined.

The related process was first introduced in the case of $\ka =2 $ and $\beta=0$ by Walsh in \cite{walsh} (see also \cite{ito}) under the name of \emph{skew} Brownian motion: it differed from the standard Brownian motion on $\R$ only in the fact that signs of its excursions from $0$ were determined by independent Bernoulli variables. In an informal description of the process generated by $\aws$ in the general case of $\ka \ge 2$ (still with $\beta=0$) we think of the graph as the spider's web, and 
of $\alpha_i$ as the probability that a spider passing through graph's center will continue its movement on the $i$th edge of the graph; $\beta$ is an additional parameter (playing a similar role to $a_i$s of \eqref{intro:1}; this fact is also reflected in \eqref{intro:2}) that tells us how sticky the graph's center is. For more on the Walsh's process see \cite{barlow,kostrykin2012,yor97,manyor}.

\newcommand{\rysunki}{
\begin{figure}
\begin{tikzpicture}
\begin{scope}[scale=0.5]      
  \foreach \n in {1,...,3}{
  \draw [blue] (0,\n) -- (5,\n); \draw [blue] (5,\n) -- (5.5,\n);
}  \draw [black,dotted] (0,1) to [out=180,in=180] (0,2);   \draw [black,dotted] (0,2) to [out=180,in=180]  (0,3); \draw [black,dotted] (0,1) to [out=180,in=180] (0,3);  \draw [black,thick,->] (6.5,2) -- (9,2); \node [above] at (7.75,2) {$\eps \to 0$}; \end{scope}
\begin{scope}[shift={(6,1)}]
    \node[circle,fill=blue,inner sep=0pt,minimum size=2pt] at (360:0mm) (center) {};
    \foreach \n in {1,...,3}{
        \node [circle,fill=blue,inner sep=0pt,minimum size=0.1pt] at ({\n*360/3}:1.5cm) (n\n) {};
        \draw [blue](center)--(n\n);}
     \node[circle,fill=blue,inner sep=0pt,minimum size=2pt] at (360:0mm) (center) {}; \end{scope}
      \end{tikzpicture}
\caption{\footnotesize{State-space collapse. As permeability coefficients $c_i$ become infinite, times spent at the points $(i,0), i \in \mc K$ before jumps become shorter and shorter. As a result, in the limit all these points are lumped together and $S_\ka$ becomes $\kak$ (here $\ka =3$).}}\label{rys3} 
\end{figure} 

\begin{figure}
\begin{tikzpicture}
\begin{scope}[scale=0.5]      
  \foreach \n in {0,...,4}{
  \draw [blue] (0,\n) -- (5,\n); \draw [blue] (5,\n) -- (5.5,\n);
}  \draw [black,dotted] (0,1) to [out=180,in=180] (0,2);   
 \draw [black,dotted] (0,0) to [out=180,in=180] (0,1);
  \draw [black,dotted] (0,3) to [out=180,in=180] (0,4);  
\draw [black,dotted] (0,2) to [out=180,in=180]  (0,3);
 \draw [black,dotted] (0,1) to [out=180,in=180] (0,3); 
  \draw [black,dotted] (0,0) to [out=180,in=180] (0,3); 
    \draw [black,dotted] (0,0) to [out=180,in=180] (0,2); 
        \draw [black,dotted] (0,0) to [out=180,in=180] (0,4);
         \draw [black,dotted] (0,1) to [out=180,in=180] (0,4);  
                  \draw [black,dotted] (0,2) to [out=180,in=180] (0,4);  
  \draw [black,thick,->] (6.5,2) -- (9,2); \node [above] at (7.75,2) {$\eps \to 0$}; \end{scope}
\begin{scope}[shift={(6,1)}]
    \node[circle,fill=blue,inner sep=0pt,minimum size=2pt] at (360:0mm) (center) {};
    \foreach \n in {1,...,5}{
        \node [circle,fill=blue,inner sep=0pt,minimum size=0.1pt] at ({\n*360/5}:1.6cm) (n\n) {};
        \draw [blue](center)--(n\n);}
     \node[circle,fill=blue,inner sep=0pt,minimum size=2pt] at (360:0mm) (center) {}; \end{scope}
      \end{tikzpicture}
\caption{\footnotesize{State-space collapse. As permeability coefficients $c_i$ become infinite, times spent at the points $(i,0), i \in \mc K$ before jumps become shorter and shorter. As a result, in the limit all these points are lumped together and $S_\ka$ becomes $\kak$ (here $\ka =3$).} }\label{rys3} 
\end{figure} }

\begin{figure}
\begin{tikzpicture}
\begin{scope}[scale=0.5]      
  \foreach \n in {0,...,4}{
  \draw [blue] (0,\n/2+1) -- (5.4,\n);
} 
 \draw [black,densely dotted] (0,1.5) to [out=180,in=180] (0,2);   
  \draw [black,densely dotted] (0,1.5) to [out=180,in=180] (0,1);   
   \draw [black,densely dotted] (0,1.5) to [out=180,in=180] (0,3);    \draw [black,densely dotted] (0,1.5) to [out=180,in=180] (0,2.5);   
    \draw [black,densely dotted] (0,1) to [out=180,in=180] (0,2.5);    \draw [black,densely dotted] (0,2) to [out=180,in=180] (0,2.5);  
     \draw [black,densely dotted] (0,2.5) to [out=180,in=180] (0,3);     \draw [black,densely dotted] (0,1) to [out=180,in=180] (0,3);   
     \draw [black,densely dotted] (0,2) to [out=180,in=180] (0,3);   
 \draw [black,densely dotted] (0,1) to [out=180,in=180] (0,2);   
  \draw [black,thick,->] (6.5,2) -- (9,2); \node [above] at (7.75,2) {$\eps \to 0$}; \end{scope}
\begin{scope}[shift={(6,1)}]
    \node[circle,fill=blue,inner sep=0pt,minimum size=2pt] at (360:0mm) (center) {};
    \foreach \n in {1,...,5}{
        \node [circle,fill=blue,inner sep=0pt,minimum size=0.1pt] at ({\n*360/5}:1.6cm) (n\n) {};
        \draw [blue](center)--(n\n);}
     \node[circle,fill=blue,inner sep=0pt,minimum size=2pt] at (360:0mm) (center) {}; \end{scope}
      \end{tikzpicture}
\caption{\footnotesize{State-space collapse. As permeability coefficients $c_i$ become infinite, times spent at the points $(i,0), i \in \mc K$ before jumps become shorter and shorter. As a result, in the limit all these points are lumped together and $S_\ka$ becomes $\kak$ (here $\ka =5$).} }\label{rys3} 
\end{figure}

\subsection{The main result}
Let $\mathpzc p$ be a fixed set of parameters for $\aso$, and for $\eps >0 $ let $\mathpzc p(\eps)$ be the same set with  with $c_i$s replaced by $\eps^{-1}c_i, i\in \mc K$.   Our first limit theorem (Theorem \ref{conv}) says that the semigroups generated by $\asoe$ converge, as $\eps \to 0$, to the semigroup of sticky Walsh's process with parameters
\begin{equation}\label{intro:2} \alpha_i \coloneqq d b_ic_i^{-1}, i \in \mc K \mquad { and } 
\beta\coloneqq d \sumj a_jc_j^{-1} \end{equation}  where $ d\coloneqq \frac 1{\sumj (a_j+b_j)c_j^{-1}}$ is a normalizing constant; 
this set of parameters will be denoted $\mathpzc {q(p)}$.
In the case of $\ka=2$ and $a_1=a_2=0$ the result described above has been proved in \cite{tombatty}, and later reproduced in \cite{knigazcup}; see also \cite{aberek} for a recent continuation.

To explain the meaning of the theorem we note first that the state-space of the snapping-out Brownian motion  can also be thought of as the $\kak$ graph, provided that we imagine that at the graph's center there is a multi-faceted, semipermeable membrane, and thus we distinguish between positions `virtually at the graph's center' but on the $i$th edge, $i\in \mc K$ (we are thus doing a reverse process to that of lumping points).  In this interpretation, $c_i$ is the permeability coefficient, telling us how quickly a particle diffusing on the $i$th edge can filter through the membrane to continue its chaotic motion on the other side. By replacing $c_i$ by $\eps^{-1}c_i$ for all $i\in \mc K$ and letting $\eps \to 0$ we make the membrane completely permeable, and thus the limit process's state-space reduces to $\kak $ (see Figure \ref{rys3}). Formulae \eqref{intro:2} show that despite the apparent absence of the membrane, in the limit there remains some kind of asymmetry in the way particles approaching the graph's center from different edges pass through this center. Thus, Walsh's sticky process can be seen as a process with a trace of semipermeable membrane at the graph's center.  

On the more technical side, Theorem \ref{conv} says that the semigroups $\semasoe$, which are defined on 
$C(S_\ka)$, converge only on $C(\kak)\subset C(S_\ka)$. In Sections \ref{sec:coo} and \ref{sec:coc}, devoted to the case of $a_i =0, i \in \mc K$, we complement this theorem with information on convergence outside of $C(\kak)$, and on convergence of the related cosine families --- see  Theorems \ref{coo:thm1} and \ref{coc:thm1}, and Corollary \ref{coc:cor1}.  An estimate of the growth of stochastic matrices involved constitutes a key element in this analysis.

The main results are preceded with Section \ref{sec:tgt} where, as a preparation, we prove two generation theorems, and  
 Section \ref{sketch},  where they are put into the perspective of the general theory of convergence of semigroups and cosine families. 

\section{Two generation theorems}\label{sec:tgt}
In this section we show that operators defined in Introduction are indeed Feller generators; in particular, we compute their resolvents which will constitute a key to our main limit theorem. We start with a linear algebra lemma.

\begin{lem}\label{lemik} Let $A_i,B_i,C_i, i \in \mc K$ be given constants such that $A_i>0$. Then, for any $\eps > 0$ there is precisely one solution $(D_i(\eps))_{i\in \mc K}\in \R^\ka $ to the system      
\begin{equation}\label{tgt:1}
\eps A_i D_i(\eps) = \eps B_i +  \textstyle {\frac 1{\ka -1 }} {\textstyle\sum\limits_{j \neq i }}(C_j +D_j(\eps) ) - C_i - D_i(\eps), \qquad i \in \mc K. \end{equation}
Moreover, the limits $\displaystyle \grae D_i(\eps), i \in \mc K$ exist and are finite. \end{lem}

\begin{proof} The idea is to rewrite \eqref{tgt:1} so that uniqueness and convergence of $D_i(\eps)$ becomes evident. To this end, we let
\begin{equation}\label{tgt:2} D(\eps) = \sumi A_i D_i(\eps )\end{equation}
and note that, by \eqref{tgt:1}, 
\( D(\eps )= \sumi B_i. \)
It follows that $D(\eps)$ in fact does not depend on $\eps$ and we will write simply $D$ instead.

Without loss of generality, we assume from now on that 
\( A_\ka= \max_{i\in \mc K} A_i.\)
 Then, by substituting $\frac D{A_\ka} - \sum_{i\not = \ka} \frac {A_i}{A_\ka} D_i(\eps)$ for $D_\ka(\eps)$ into \eqref{tgt:1}, we obtain the following system of relations 
 \begin{align}
\label{tgt:3} \Big(1+ \eps A_i + \tfrac {A_i}{(\ka -1)A_\ka} \Big) D_i(\eps)&= \eps B_i +   \textstyle {\frac 1{\ka -1 }} {\displaystyle\sumjni} C_j - C_i +    \tfrac D{(\ka -1)A_\ka} \\
 &\phantom{=}+   \textstyle {\frac 1{\ka -1 }} {\displaystyle \sum_{j\not =i,k}} \Big(1- {\textstyle \frac {A_j}{A_\ka}}\Big) D_j (\eps), \nonumber \end{align}
involving variables with the first $\ka -1$ indexes \[ i \in \mc L\coloneqq \{1,\dots, \ka -1 \}. \]
This system is easier to handle. To see this, we equip  $\R^{\ka -1 }$ with the norm $\|(\xi_i)_{i\in \mc L}\|\coloneqq \max_{i\in \mc L} |\xi_i|.$ Then, for any $\eps \ge 0 $, the norm of the linear operator $O_{\eps}:\R^{\ka -1 } \to \R^{\ka -1 }$ given by 
\[ O_{\eps} (\xi_i)_{i\in L} = \Big({\textstyle\frac 1{\ka -1 }\frac 1{m_i(\eps)} {\displaystyle\sum_{j\in \mc L\setminus\{i\}}}} 
\Big (1- {\textstyle \frac {A_j}{A_\ka}}\Big) \xi_j \Big)_{i\in L}\]
where 
\( m_i(\eps )\coloneqq 1+ \eps A_i + {\textstyle \frac {A_i}{(\ka -1) A_\ka}}>1,\)
is smaller than $1$. For, $ \|O_{\eps} \| $ does not exceed
\begin{align*}\max_{i\in \mc L} {\textstyle\frac 1{\ka -1 }\frac 1{m_i(\eps)}} \sum_{j\in \mc L\setminus\{i\}} 
 \Big(1- {\textstyle \frac {A_j}{A_\ka}}\Big)<  \max_{i\in \mc L} {\textstyle\frac 1{\ka -1 }} \sum_{j\in \mc L\setminus\{i\}} 
 \Big(1- {\textstyle \frac {A_j}{A_\ka}}\Big)< {\textstyle \frac {k-2}{\ka -1 }} < 1.\end{align*}
Hence, $I - O_{\eps}$ is invertible and so \eqref{tgt:1} has the unique solution  
\begin{equation}
\left ( D_i(\eps) \right )_{i\in L} = (I - O_{\eps})^{-1} \left ( E_i(\eps) \right )_{i\in L} \label{tgt:4}\end{equation}
where
\( E_i(\eps ) = \eps B_i +   \textstyle {\frac 1{\ka -1 }} {\sumjni} C_j - C_i + \tfrac D{(\ka -1)A_\ka}, 
 i\in \mc L\). Moreover, we have $\grae O_{\eps} = O_{0}$ and
\( \grae  E_i(\eps ) =- C_i +    \tfrac 1{\ka -1 } {\sumjni} C_j +    \tfrac 1{\ka -1 } \tfrac D{A_\ka}, i \in \mc K \). Hence,
 \eqref{tgt:4} establishes convergence of $\left (D_i(\eps)\right )_{i\in L}$, as $\eps \to 0$. Since the sum in \eqref{tgt:2} does not depend on $\eps$ and $A_i$s are nonzero, this implies convergence of all $D_i(\eps)$s and thus completes the proof.
  \end{proof}

\begin{prop}\label{tgt:prop1} For any set $\mathpzc p$ of non-negative parameters $a_i,b_i,c_i$ with $b_i>0$, the operator \emph{$\aso$} of Section \ref{sec:sso} is a Feller generator. \end{prop}
\begin{proof} $\aso$ is obviously densely defined, and arguing as, for example, in \cite{konkusSIMA}*{Sec. 6} we find out that this operator satisfies the positive-maximum principle. Hence, by \cite{kallenbergnew}*{Thm 19.11} or \cite{kniga}*{Thm. 8.3.4}, we are to check only that for any $g\in C(S_\ka)$ and $\lam >0$ there is an $f\in \dom{\aso}$ solving 
\( \lam f - \aso f = g, \)
that is, for any $\cgk\in \F$ and $\lam >0$ there is an $\cfk\in \dom{\aso}$ such that 
\( \lam f_i -  f_i'' = g_i, i \in \mc K\). 
(Such an $\cfk$ is unique, because $\aso$ satisfies the positive-maximum principle and is thus dissipative --- see \cite{ethier}*{Lemma 2.1, p. 165}.) 

To this end, we search for $f_i$s of the form 
\begin{equation}\label{tgt:6} f_i(x) \coloneqq C_i \e^{\slam x}+D_i \e^{-\slam x} - \tfrac{1}{\slam} \int_0^x \sinh \slam (x-y) g_i(y) \ud y, \ \  x \ge 0, i\in \mc K, \end{equation}
for some constants $C_i,D_i, i \in \mc K$, and note that the limits $\gra f_i(x)$ exist and are finite iff
\begin{equation} \label{tgt:7} C_i = \tfrac 1{2\slam}  \int_0^\infty \e^{-\slam y} g_i(y) \ud y, \qquad i \in \mc K.\end{equation}
Moreover, condition \eqref{intro:1} is satisfied iff 
\begin{equation}\label{tgt:8}
\gamma_i^+ D_i = \gamma_i^-C_i + a_i c_i^{-1} g_i(0) +  \textstyle {\frac 1{\ka -1 }} {\textstyle\sum\limits_{j \neq i }}(C_j +D_j ) - C_i - D_i, \qquad i \in \mc K\end{equation}
where
\( \gamma_i^\pm \coloneqq b_ic_i^{-1} \slam \pm \lam a_ic_i^{-1},  i \in \mc K. \) Since \eqref{tgt:8} is obviously a particular case of \eqref{tgt:1} with $\eps =1$, $A_i=\gamma_i^+$ and $B_i =  \gamma_i^-C_i + a_i c_i^{-1} g_i(0)$, existence of (unique) solution to this system is guaranteed by Lemma \ref{lemik}.
\end{proof}

Our second generation theorem is a simple case of the general result of \cite{kostrykin2012}*{Thm. 2.8} and \cite{kostrykin2}*{Thm. 3.7}, but we sketch its proof here since in what follows we need the form of the resolvent of $\aws$.

\begin{prop}\label{tgt:prop2}Operator $\aws$ is a Feller generator.\end{prop}
\begin{proof} We proceed as in the proof of the previous proposition. First of all,  density of $\dom{\aws}$ and the positive-maximum principle for $\aws$ do  not pose a problem. Secondly, to find a solution to the resolvent equation we search for $f_i$s of the form \eqref{tgt:6} (with $C_i$s defined by \eqref{tgt:7}), and note that the defining conditions (b) and (c) are satisfied iff 
\begin{equation}\label{sec1:3} 
f_i(0)=f_j(0) \mquad{that is} C_i +D_i = C_j +D_j, \qquad i,j\in \mc K\end{equation}
and
\begin{equation}\label{sec1:4} 
\beta [\lam (C_i+D_i) - g(0) ] = \slam \sumj \alpha_j (C_j- D_j), \qquad i \in \mc K,\end{equation}  
respectively. Denoting by $f(0)$ the common value of \eqref{sec1:3} we check that the system \eqref{sec1:3}--\eqref{sec1:4} has the following unique solution: \(
D_i = f(0) - C_i, i \in \mc K \) where \(
f(0) = \frac {\beta  g(0) + 2\slam \sumi \alpha_i C_i}{\lam \beta + \slam (1-\beta)}\).
This completes the proof. \end{proof}

\section{The main convergence theorem}
\subsection{A sketch of the theory of convergence of semigroups and bird's-eye view of our results}\label{sketch}
The main idea of the Trotter--Kato--Neveu convergence theorem \cite{abhn,engel,goldstein,pazy}, a cornerstone of the theory of convergence of semigroups \cite{knigazcup,bobrud}, is that convergence of resolvents of \emph{equibounded} semigroups in a Banach space $\mathsf F$, gives an insight into convergence of the semigroups themselves. Hence, in studying the limit of equibounded semigroups, say, $\sem{B_\eps}$, generated by the operators $B_\eps, \eps >0$ we should first establish existence of 
the strong limit 
\[ \rla \coloneqq \grae \rez{B_\eps}. \]
The general theory of convergence (see \cite{knigazcup}*{Chapter 8}) covers also the case in which, unlike in the  classical version of the Trotter--Kato--Neveu theorem, the (common) range of the operators $\rla, \lam >0$  so-obtained is \emph{not} dense in  $\mathsf F$, and stresses the role of  the so-called \emph{regularity space}, defined as the closure of the range of $\rla$: 
\begin{equation*} \freg \coloneqq cl (Range \rla) \subset \mathsf F. \end{equation*}
Namely, $\freg$ turns out to coincide with the set of $f\in \mathsf F$ such the limit $T(t)f \coloneqq \grae \e^{tB_\eps}f $ exists and is uniform with respect to $t$ in compact subintervals of $[0,\infty)$; then $\{T(t),\, t \ge 0\}$, termed the \emph{regular limit} of $\sem{B_\eps}, \eps >0$, is a strongly continuous semigroup in $\freg$.  

It should be stressed, though, that this statement does not exclude the possibility of the existence of $f\not \in \freg$ such that the strong limit $\grae \e^{tB_\eps}f$  exists for all $t\ge0$. Such \emph{irregular} convergence of semigroups, which is known to be always uniform with respect to $t$ in compact subsets of $(0,\infty)$ --- see \cite{note} or \cite{knigazcup}*{Thm 28.4} --- is not so uncommon, especially in the context of singular perturbations \cite{banmika,banalachokniga,knigazcup}, but needs to be established by different means.

In our first limit theorem (Theorem \ref{conv}) we fix the set $\mathpzc p$ of parameters $a_i,b_i,c_i$ of Section \ref{sec:sso} and consider operators $\asoe, \eps >0$ defined as $\aso$ with $\mathpzc p =\mathpzc p(\eps)$ obtained from $\mathpzc p$ by replacing all $c_i$s by $\eps^{-1}c_i$ and leaving the remaining parameters intact.  We prove that the regularity space for this family of semigroups of operators is $C(\kak)\subset C(S_\ka)$, and that the sticky Walsh's process is their regular limit. In Section \ref{sec:coo}, under additional assumption that all $a_i$ are zero, we will show $\asoe$ generate also equibounded cosine families, and therefore the limit $\grae \e^{t\asoe}f, t \ge 0$  exists also for $f \in C(S_\ka) \setminus C(\kak)$.  In Section \ref{sec:coc} we show that convergence of semigroups (and cosine families) on the regularity space is in fact uniform with respect to $t$, and that outside of this subspace the cosine families do not converge at all.

\subsection{The main theorem}
Let $\mathpzc p$ and $\mathpzc p (\eps), \eps >0$ be as in the preceding  section. In other words, elements of the domain of $\asoe$ satisfy 
the transmission conditions 
\begin{equation}\eps a_if_i''(0) - \eps b_i f_i'(0) = c_i  \Big [{ \textstyle \frac 1{\ka -1 }} {\sum\limits_{j \neq i}}f_j(0) - f_i(0) \Big], \qquad i \in \mc K, \label{tmt:1}\end{equation}
and we are interested in the strong limit of the semigroups $\semasoe$ generated by $\asoe,$ as $\eps \to 0$. 
As explained above, our first task is to prove existence of the limit of resolvents of $\asoe$. This is achieved in the following proposition. 

\begin{prop} \label{tmt:prop1} The limit \( \rla g \coloneqq \grae \rezasoe g\) exists for all $ g \in C(S_\ka)$ and $ \lam >0$, and the regularity space  for $\semasoe, \eps >0$ coincides with $C(\kak)$. Moreover, for $g \in C(\kak)$, $\rla g = \rezw g$ where $\mathpzc{q(p)}$ is the set of parameters given in  \eqref{intro:2}.   \end{prop} 

\begin{proof} Fix $g\in C(S_\ka)$. Function $f_\eps =\rezasoe g$ solves the resolvent equation for $\asoe$ and hence, as we know from the proof of Proposition \ref{tgt:prop1}, is of the form \eqref{tgt:6}--\eqref{tgt:7} with certain $D_i =D_i(\eps)$. More specifically, the vector of these coefficients is a unique solution to the system 
\begin{equation}\label{tmt:2}
\eps \gamma_i^+ D_i(\eps) = \eps \gamma_i^-C_i + \eps a_i c_i^{-1} g_i(0) +  \tfrac 1{\ka -1 } {\textstyle\sum\limits_{j \neq i }}(C_j +D_j(\eps) ) - C_i - D_i(\eps),\end{equation}
where, as in Proposition \ref{tgt:prop1},
\( \gamma_i^\pm \coloneqq b_ic_i^{-1} \slam \pm \lam a_ic_i^{-1}, i \in \mc K. \) We are thus dealing again with a  special case of \eqref{tgt:1}, and Lemma \ref{lemik} tells us that  the limits $D_i^0 \coloneqq \grae D_i(\eps), i \in \mc K$ exist and are finite. Since neither the first nor the third term in \eqref{tgt:6} depend on $\eps$, this establishes existence of the limit $\grae \rezasoe g$.  

To prove the second sentence in the proposition, we first let  $\eps \to 0$ in \eqref{tmt:2} to obtain
\begin{equation}\label{tmt:3} C_i +D_i^0 = \tfrac 1{\ka -1  }\sum_{j\not =i} (C_j+D_j^0), \qquad i \in \mc K. \end{equation}
This shows that  $f_i(0)=C_i+D_i^0$ does not depend on $i\in \mc K$, that is, that $\rla g $ belongs to $C(\kak)$. Next, we sum both sides of \eqref{tmt:2} over $i\in \mc K$ and divide by $\eps$.  This renders
\[ \sumi \big(b_ic_i^{-1} \slam + \lam a_ic_i^{-1} \big)D_i^0 =  \sumi \big(b_ic_i^{-1} \slam - \lam a_ic_i^{-1} \big)C_i
+ \sumi  a_ic_i^{-1} g_i(0).\]
If $g\in C(\kak)$, that is, if $g_i(0)$ does not depend on $i$, this relation can be rearranged as
\[ \sumi ( a_ic_i^{-1}) (\lam (C_i+D_i^0) - g (0)) = \slam \sumj  b_jc_j^{-1} (C_j-D_j^0). \] 
This means, however, that condition \eqref{sec1:4} is satisfied with $\alpha_i$s and $\beta$ specified in \eqref{intro:2}, because we know that $C_i +D_i^0$ does not depend on $i$. It follows that for the stated choice of parameters, $D_i$s of \eqref{sec1:3}--\eqref{sec1:4}, coincide with $D_i^0$s. This establishes $\rla g = \rezw g $ for $g\in C(\kak)$. 

As a by-product, the range of $\rla$ contains $\dom{\awsp}$, and since the latter is dense in $C(\kak)$ the closure of the range of $\rla$ contains $C(\kak)$. But we have already established that this range is contained in $C(\kak)$. Hence, the closure of the range of $\rla$ coincides with $C(\kak),$ as claimed.  \end{proof}



Proposition \ref{tmt:prop1}  says in fact more than it is visible on its surface. To wit, the fact that on the regularity space $C(\kak)$, $\rla $ coincides with $\rezw$ implies that $\awsp$ is the generator of the regular limit of $\semasoe$ --- see e.g. \cite{knigazcup}*{Thm. 8.1 and Corollary 8.3}, compare  \cite{kniga}*{Sections 8.4.3 and 8.4.4}. As an immediate corollary we obtain thus the following main theorem of this section.

\begin{thm}\label{conv} We have
 \[ \grae \e^{t\asoe} f = \e^{t\awsp} f, \qquad f \in C(\kak) \]
 with the limit uniform with respect to $t$ in compact subsets of $[0,\infty)$. 
\end{thm}

In view of the Trotter--Sova--Kurtz--Mackevi\u cius theorem \cite{kallenbergnew}, this result can be seen as expressing a convergence of the random processes involved. We note that A. Gregosiewicz in \cite{adas} has found a different proof of our theorem based on a decomposition of resolvents involved.

\section{Convergence outside of $C(\kak)$}\label{sec:coo}
\subsection{Introductory remarks}\label{sec:ir}
As already mentioned in Section \ref{sketch}, Theorem \ref{conv} need not tell the entire story: it may happen that $\grae \e^{t\asoe}f $ exists also for $f\in C(S_\ka)\setminus C(\kak)$, except that this limit cannot be uniform in compact subintervals containing $0$.  Such convergence of semigroups outside of regularity space can be deduced from the convergence of their resolvents provided that the semigroups enjoy additional regularity properties, such as being uniformly holomorphic --- see e.g. \cite{knigazcup}*{Chapter 31}. In \cite{aberman} it has been proved that, in the case of $\ka =2 $ and $a_1=a_2=0$, each $\aso$ is a cosine operator family generator, and all these families are formed by operators of norm not exceeding $5$. As a corollary,  in \cite{aberek} we show that in this case convergence spoken of in Theorem \ref{conv} extends beyond $f \in C(\kak)$.  

It is the purpose of this section to prove a similar result for general $\ka \ge 3$. Hence, in what follows we assume that $a_i=0$, $i\in \mc K$ and, without loss of generality, take $b_i=1, i\in \mc K,$ so that the entire generator 
$\aso$ is characterized by the vector $\mathpzc c= (c_1,\dots, c_\ka)\in \mathbb R^\ka$ of (positive) permeability coefficients; to stress this in what follows we write $\asoc$ instead of $\aso$.  As a result, transmission conditions \eqref{intro:1} take the form
 \begin{equation} f_i'(0) = c_i  \Big[ f_i(0)-  {{\textstyle \frac 1{\ka -1 }} \sum\limits_{j \neq i}} f_j(0)  \Big ], \qquad i \in \mc K.\label{coo:1} \end{equation}
We will show (see Theorem \ref{coo:thm1}) that each $\asoc$ generates a strongly continuous cosine family  $\{\Cos_{\asoc} (t), t \in \R\}$ such that 
\[ \|\Cos_{\asoc} (t)\| \le M = M(\mathpzc c), \qquad t \in \mathbb R, \] 
where $M(\mathpzc c)$ has the following property: for any $r>0$, $M(r\mathpzc c) = M(\mathpzc c)$. In particular, even though we cannot claim that (as in the case of $\ka =2$) there is a universal bound for all cosine families generated by $\asoc$s, we know that for fixed $\mathpzc c$ the operators $\asoce $, defined as $\asoc $ with $\mathpzc c$ replaced by $\eps^{-1} \mathpzc c$, generate cosine families that are bounded by a universal constant: 
\begin{equation}\label{coo:2} \|\Cos_{\asoce} (t)\| \le M(\mathpzc c), \qquad t\in \R, \eps >0.\end{equation}
Theorem \ref{conv} says that  the semigroups generated by operators $\asoce$ converge, as $\eps\to 0$, to the semigroup describing the Walsh's process with  parameters \begin{equation}\label{coo:3} \alpha_i \coloneqq \tfrac 1{c_i\sumj c_j^{-1}},\; i \in \mc K \mquad { and }  \beta\coloneqq 0.\end{equation}
To denote this simpler set of parameters we write $\alpha (\mathpzc c)$ instead of $\mathpzc {q(p)}$.

As explained in detail in \cite{aberek}, estimate \eqref{coo:2}, when combined with Proposition \ref{tmt:prop1}, implies 
the following second main result of our paper. 
\begin{thm} \label{coo:thm1}  \  
\begin{itemize}
\item [(a) ] For $f\in C(\kak)$, $\grae \Cos_{\asoce} (t) f = \Cos_{\awsa} (t)f$ uniformly in $t\in [0,t_0]$ for $t_0>0$.   
\item [(b) ] For $f \in C(S_\ka) \setminus C(\kak)$, $\grae \e^{t\asoce} f =\e^{t\awsa}f $  
  uniformly in $t \in (t_0^{-1}, t_0)$ for $t_0 >1$.   \end{itemize} 
\end{thm}
To elaborate on these succinct statements: first of all, $\awsa$ is also the generator of a cosine family in $C(\kak)$ (see our Section \ref{sec:coc} for more details). As such, despite not being densely defined in $C(S_\ka)$, it is also the generator of a semigroup $\{\e^{t\awsa},t \ge 0\}$ of operators in $C(S_\ka)$ --- see e.g. \cite{odessa}*{Corollary 5.1}. These operators are extensions of those defined in $C(\kak)$, and the semigroup generated by $\awsa$ is not strongly continuous in $t\in [0,\infty)$ but merely in $t\in (0,\infty)$ (of course, for $f\in C(\kak), \lim_{t\to 0+}\e^{t\awsa}f = f$). This clarifies statement (b). 

Condition (a) also requires a comment. Since we were able to extend convergence of semigroups from $C(\kak)$ to the entire $C(S_\ka)$, it may seem natural to ask weather the same can be done with convergence of the related cosine families. However, as proved in \cite{zwojtkiem} (see also \cite{knigazcup}*{Chapter 61}), cosine families by nature cannot converge outside of the regularity space. As applied to our case, this theorem tells us that (a) is the best result possible in the sense that  the limit cosine family cannot  be extended beyond $C(\kak)$ --- see also our Section \ref{sec:coc}.

Finally, we remark that the fact that each $\asoc$ is the generator of a cosine family can also be proved using decomposition of resolvent techniques (see \cite{adas}); however, this approach does not give a universal bound for the norm of $\Cos_{\asoce}(t), t\in \R, \eps >0$ (found in \eqref{coo:2}).

\subsection{Definition of $M(\mathpzc c)$} Given $\mathpzc c = (c_1,\dots, c_\ka)$  with all $c_i>0$ we think of the $\ka \times \ka$  intensity matrix  $Q=(q_{i,j})_{i,j\in \mc K}$ given by 
\begin{equation}\label{coo:4}
q_{i,j}= \begin{cases} -c_i,& i=j, \\  \frac{c_i}{\ka -1 },&  i\neq j
.\end{cases}
\end{equation}
The Markov chain  generated by  $Q$ is irreducible and its invariant measure is
$\alpha=(\alpha_i)_{i \in \mc K}$, where $\alpha_i$ are defined in \eqref{coo:3}. 
 
Let $c\coloneqq \max_{i \in \mc K} c_i$ and $Q_0\coloneqq c^{-1}Q$. We denote the entries of the matrix $\e^{tQ_0}$ by $p_{i,j}^0(t)$. 
Since $ Q_0+I_\ka$, where $I_\ka$ is the $\ka\times\ka$ identity matrix,  is a transition matrix of an irreducible and reversible discrete time Markov chain with the invariant measure $\alpha$, Corollary 2.1.5 in \cite{spectralgap} implies that 
\begin{align}
|p_{i,j}^0(t)-\alpha_j|\leq \e^{-\omega t} {\textstyle\sqrt{\frac{\alpha_j}{\alpha_i}}}= \e^{-\omega t}{\textstyle \sqrt{\frac{c_i}{c_j}}}, \qquad t\geq 0, \; i,j \in \mc K,\label{coo:5}
\end{align}
where the \emph{spectral gap} $\omega$ is the smallest non-zero eigenvalue of $-Q_0$. We note that $Q_0$ does not change if $\mathpzc c$ replaced by $r \mathpzc c$ where $r>0$, and thus neither does change the $\omega$.  It follows that 
the same applies to the  constant 
\begin{align}\label{coo:6}
M= M(\mathpzc c) \coloneqq \Big(1+ 2\max_{i \in \mc K}\frac{c_i}{c \omega } \sum_{j \in \mc K} \Big({\textstyle \sqrt{\frac{c_i}{c_j}}}+{\textstyle \frac{1}{\ka -1 }} \sum_{\ell \neq i} {\textstyle\sqrt{\frac{c_\ell}{c_j}}} \Big) \Big).
\end{align}

\subsection{Cosine families}\label{sec:cf} 
A strongly continuous family $\{C(t), t \in \R\}$ of operators in a Banach space $\mathsf F$ is said to be a cosine family iff $C(0)$ is the identity operator and 
\[ 2 C(t) C(s) = C(s+t) + C(t-s), \qquad t,s\in \R.\]
The generator of such a family is defined by 
\[ Af = \lim_{t\to 0} 2t^{-2}{(C(t)f- f)}\]
for all $f \in \mathsf F$ such the limit on the right-hand side exists. 
For example, in $\cer$, the space of continuous functions on $\R$ that have finite limits at $\pm\infty$,  there is the \emph{basic cosine family} given by  
\[ C(t) f(x) = \pol [ f(x+t) + f(x-t) ], \qquad x \in \R, t \in \R.\]
Its generator is the  one-di\-men\-sion\-al Laplace operator $f\mapsto f''$ with domain composed of twice continuously differentiable functions on $\R$ such that $f'' \in \cer$. 

Each cosine family generator is automatically the generator of a strongly continuous semigroup (but not vice versa). The semigroup  such operator generates is given by the \emph{Weierstrass formula} (see e.g. \cite{abhn}*{p. 219})
\[ T(0)f =f    \mquad {and} T(t) f  = {\textstyle \frac 1{2\sqrt{\pi t}}} \int_{-\infty}^\infty \e^{-\frac {s^2}{4t}} C(s) f  \ud s, \qquad t >0, f \in \mathsf F.\] 
This formula expresses the fact that the cosine family is thus (in this case) a more fundamental object than the semigroup, and properties of the semigroup can be hidden in those of the cosine family (see e.g. \cite{newfromold}). Moreover, semigroups that are generated by generators of cosine families are much more regular than other semigroups. In particular, such semigroups are 
holomorphic, but even among holomorphic semigroups there are those that are not generated by cosine family generators (see \cite{abhn} again).

\subsection{Lord Kelvin's method of images and the generation theorem}\label{gener_result}
For a number of Laplace operators in $\cerp$ with domains characterized by Feller--Wentzel boundary conditions at $x=0$ the cosine families they generate can be constructed semi-explicitly as  (isomorphic images of) subspace cosine families for the basic cosine families in $\cer$ --- see \cite{kosinusy,aberman} and references given there.  The trick, known as the \emph{Lord Kelvin method of images} \cite{kosinusy,feller}, comes down to noticing that each boundary condition unequivocally shapes extensions of elements of $\cerp$ to elements of $\cer$; for example, the Neumann and Dirichlet boundary conditions lead to even and odd extensions, respectively. These extensions form an invariant subspace for the basic cosine family, and the cosine family we are searching for turns out to be an isomorphic image of the basic cosine family as restricted to this subspace.

In this section we use the same idea to show that the operator $\asoc$  with transmission conditions \eqref{coo:1} generates a cosine family $\{\Cos_{\asoc}(t),\, t \in \R\} $  of operators  in $\F$: we will construct $\{\Cos_{\asoc}(t),\, t \in \R\} $ as an isomorphic image of a subspace cosine family of the  \emph{Cartesian product cosine family} $ \{\ced(t), t\in \R\} $ (`D' for `Descartes'). The latter is defined in the Cartesian product  space
\[ \G \coloneqq (\cer )^\ka , \]
(equipped with the maximum norm) by the formula 
\[ \ced (t) (f_i)_{i \in \mc K} = (C(t)f_i)_{i \in \mc K}, \qquad  \cfk \in \G, t\in \R .\]
This is to say that $\Cos_{\asoc} (t)$  will be found to be of the form 
\begin{equation} \Cos_{\asoc} (t)\cfk  = R\ced(t)\wcfk,  \qquad  \cfk\in \F, t \in \R,\label{coo:7} \end{equation}
where  $\wt{f_i} \in \cer, i \in \mc K$,  is a suitable extension of $f_i \in \cerp$
and  $R\colon \G \to \F$
is the  restriction operator
 assigning to a $(g_i)_{i \in \mc K} \in \G$ the member $(f_i)_{i\in \mc K}$ of $\F $ given by  $f_i = g_{i|\R_+}, i \in \mc K$. Our first lemma says that extensions $\wt{f_i}, i \in \mc K$ are determined uniquely by the fact that a cosine family leaves the domain of its generator invariant.

\begin{lem}\label{coo:lem1}
For $\cfk \in D(\asoc)$ there exists its unique extension $\wcfk \in \G $ such that 
\begin{equation*} 
R\ced(t)\wcfk \in D(\asoc) 
\end{equation*}
for $t \in \R$. Moreover, each  $\wt{f_i}, i\in \mc K$ belongs to the domain of the generator of the basic cosine family and \begin{equation} \label{coo:8} \| (\wt{f_i})_{i\in \mc K}\|_{\G} \le M \| (f_i)_{i\in \mc K}\|_{\F}, \end{equation}
where $M$ is defined in \eqref{coo:6}. \end{lem}
\begin{proof}
\textbf {Step I, existence and uniqueness of $\wt{f_i}, i\in \mc K$.} \textrm Our task is to find 
\begin{align}\label{coo:9} g_i(x)\coloneqq \wt{f_i}(-x), \qquad x\geq 0, i \in \mc K, \end{align} satisfying compatibility conditions $f_i(0)=g_i(0)$. Since $\ced(t)=\ced(-t), t\geq 0$, we must also have 
\begin{align*}
\frac{\textrm{d}}{\textrm{d}x} [\wt{f_i}(x-t)+\wt{f_i}(x+t)]_{|x=0} =c_i\Big[  \wt{f_i}(t)+\wt{f_i}(-t)-  {{\textstyle \frac 1{\ka -1 }} \sum\limits_{j \neq i}} \big(\wt{f_j}(t)+\wt{f_j}(-t)\big)\Big],
\end{align*}
that is
\begin{align}
f_i'(t)-g_i'(t) =c_i\Big[ {f_i}(t)+g_i(t)-  {{\textstyle \frac 1{\ka -1 }} \sum\limits_{j \neq i}} \big({f_j}(t)+g_j(t)\big)\Big], \label{coo:10}
\end{align}
 for $t\geq 0, i \in \mc K$.
This system can be rewritten as
\begin{align*}
(g_i-f_i)'_{i\in \mc K}=Q (g_i-f_i)_{i\in \mc K}+2Q\cfk,
\end{align*}
where $Q=(q_{i,j})_{i,j\in \mc K}$ is the intensity matrix defined in \eqref{coo:4}. 
Hence, $g_i$ are uniquely determined and (because of the compatibility condition) given by \begin{align}
(g_i(t))_{i\in \mc K}
&=(f_i(t))_{i\in \mc K}+2\int_0^t\big[\e^{(t-s)Q}Q\big](f_i(s))_{i\in \mc K}\textrm{d}s, \qquad t\geq 0. \label{coo:11}
\end{align}
From now on, we treat \eqref{coo:10} and \eqref{coo:11} as the definition of $( \wt{f_i})_{i\in \mc K}$. 


\textbf {Step II, an estimate for $p_{i,j}'(t)$.}
We note that $\e^{tQ}, t\geq 0$ is the matrix of transition probabilities, say, $p_{i,j}(t)$,  for the Markov chain with intensity matrix $Q$, whereas $Q\e^{tQ}=\e^{tQ}Q=\frac{\ud }{\ud t}\e^{tQ}= (p_{i,j}'(t))_{i,j\in \mc K}$. 
Moreover,
\begin{align*}
|p_{i,j}'(t)|&=\Big |\sum_{\ell \in \mc K} q_{i,\ell} p_{\ell,j}(t) \Big|=\Big|\sum_{\ell \neq i}\textstyle{ \frac{c_i}{\ka -1 }}p_{\ell,j}(t)-{c_i}p_{i,j}(t) \Big|\\
&\leq {\tfrac{c_i}{\ka -1 }}  \sum_{\ell \neq i} |p_{\ell,j}(t)-p_{i,j}(t)|\leq {\textstyle\frac{c_i}{\ka -1 }}   \sum_{\ell \neq i} (|p_{\ell,j}(t)-\alpha_j|+|p_{i,j}(t)-\alpha_j|) \end{align*}
for $t\geq 0$, $i, j \in \mc K$. Next, \eqref{coo:5} and $p_{i,j}(t)=p_{i,j}^0(ct), i,j \in \mc K$ imply that 
\begin{align}\label{coo:12}
|p_{i,j}(t)-\alpha_j|\leq \e^{-c\omega  t} {\textstyle \sqrt{\frac{c_i}{c_j}}}, \qquad t\geq 0, \; i,j \in \mc K.
\end{align}
This in turn renders
\begin{align}\label{coo:13}
|p_{i,j}'(t)|&= c_i\e^{-c\omega  t}\Big({\textstyle \sqrt{\frac{c_i}{c_j}}}+{\textstyle \frac{1}{\ka -1 }} \sum_{\ell \neq i} {\textstyle\sqrt{\frac{c_\ell}{c_j}}} \Big), \qquad t\geq 0,\; i,j \in \mc K.
\end{align}
The last two estimates will be of key importance in what follows.  

\textbf {Step III, each $\wt{f_i}$ belongs to $\cer$, and \eqref{coo:8} holds.} For our first claim in this step it suffices to show that $g_i \in \cerp$, and since continuity of $g_i$ is clear, it is enough to prove that the integral in \eqref{coo:11} converges, as $t\to \infty,$ to $\Pi m - m$ where $m\coloneqq (m_i)_{i\in \mc K}, m_i\coloneqq \grat f_i(t)$  and $\Pi m = (\sum_{j\in \mc K} \alpha_j m_j)_{i\in \mc K}$. Now, the last statement is true if $f_i(t)=m_i $ for $t\ge 0$ and $i\in \mc K$ because then the integral equals $\int_0^t \e^{(t-s)Q}Q m \ud s = \e^{tQ}m - m$ and \eqref{coo:12} implies that $\grat \e^{tQ} m = \Pi m$. Hence, we are left with showing that the integral in question converges to $0$ provided that $m=0$. 

To this end, we let $\R^\ka $ be equipped with the maximum norm. Then, \eqref{coo:13} shows that the norm of $\e^{tQ}Q$ as the operator in $\R^\ka$ can be estimated as follows:
\begin{equation}\label{coo:14} \|\e^{tQ}Q\| = \max_{i\in \mc K} \sum_{j\in \mc K} |p_{i,j}'(t)| \le  c M_0 \e^{-c\omega  t}  \qquad t\geq 0 
\end{equation}
where $M_0 \coloneqq \max_{i\in \mc K} \sum_{j \in \mc K} \Big({\textstyle \sqrt{\frac{c_i}{c_j}}}+{\textstyle \frac{1}{\ka -1 }} \sum_{\ell \neq i} {\textstyle\sqrt{\frac{c_\ell}{c_j}}} \Big)$.

Next, if $m=0$, given $\epsilon >0$ we can find a $t_0>0$ such that $|f_i(s)|< \epsilon, i \in \mc K $ as long as $s\ge t_0$. Hence, for $t> t_0$ the integral in question does not exceed $\int_0^{t_0} \|\e^{(t-s)Q}Q \| \, \|f\|_{\F}\ud s+ \epsilon \int_{t_0}^t  \|\e^{(t-s)Q}Q \| \ud s$. Since, by \eqref{coo:14} the first summand converges to $0$, as $t\to \infty$, and the second is bounded by $\frac {\epsilon M_0}{c\omega}$, our first claim follows.

As to the other claim,  \eqref{coo:13} implies also $
 \int_0^\infty |p_{i,j}'(t)| \textrm{d}t\leq {\frac{c_i}{c\omega}}\Big({\textstyle \sqrt{\frac{c_i}{c_j}}}+{\textstyle \frac{1}{\ka -1 }} \sum_{\ell \neq i} {\textstyle\sqrt{\frac{c_\ell}{c_j}}} \Big)
$. Thus, \eqref{coo:11} combined with  
\(
\int_0^t\big[\e^{(t-s)Q}Q\big](f_i(s))_{i\in \mc K}\textrm{d}s=\Big(\sum_{j \in \mc K} \int_0^t p_{i,j}'(t-s)f_j(s)\textrm{d}s\Big)_{i \in \mc K}
\)
shows that 
\begin{align}\|(g_i)_{i \in \mc K}\|\leq \Big(1+ 2\max_{i \in \mc K} \sum_{j \in \mc K}    \int_0^\infty |p_{i,j}'(t)|\textrm{d}t\Big) \le M \|\cfk\|. 
\end{align}
Hence,  \eqref{coo:8}  is established. 

\textbf {Step IV, $\wt{f_i}$ is in the domain of the generator of the basic cosine family.}
Since $f_i$ are twice continuously differentiable on $[0,\infty)$, \eqref{coo:11} shows that  so are  $g_i$,  $i\in \mc K$. Also, by differentiating \eqref{coo:11} we recover \eqref{coo:10}, which in turn, when evaluated at $t=0$ yields 
\(
g_i'(0) =f_i'(0) - 2 c_i\Big[ f_i(0)-  {{\textstyle \frac 1{\ka -1 }} \sum_{j \neq i}} \big({f_j}(t)\big)\Big]=-f_i'(0), i \in \mc K.
\)
This proves that the left-hand and right-hand derivatives of $\wt{f_i}$ agree at $t=0$, and we see  from \eqref{coo:10} that $\wt{f_i}$ is continuously differentiable on the entire $\R$.  The same relation reveals, furthermore, that $g_i'$ is continuously differentiable on $[0,\infty)$, since so is $f_i'$, and a little calculation using the already established $f'_i(0)=-g'_i(0)$ yields $f''_i(0)=g''_i(0)$ completing the proof that $\wt{f_i}$ is twice continuously differentiable.   

We are left with proving that the limits $\grat g_i''(t)$ exist and are finite. To this end, we note first that existence of finite limits $\grat f_i''(t)$ implies that $\grat f_i'(t) =0, i \in \mc K$. Furthermore, since, as we have established in Step III, $\grat (g_i(t))_{i\in \mc K} = m + 2(\Pi m - m)$, and \eqref{coo:10} can be rewritten as $ (g_i'(t))_{i\in \mc K} = (f_i'(t))_{i\in \mc K} - Q [(g_i(t))_{i\in \mc K}+ (f_i(t))_{i\in \mc K}]$, we obtain $\grat (g_i'(t))_{i\in \mc K}  = -2Q\Pi m =0$. Hence, differentiating \eqref{coo:10} once again shows that $\grat g_i''(t) $ exists and equals $\grat f_i''(t) $. \end{proof}


Lemma \ref{coo:lem1} tells us in particular that if \eqref{coo:7} is to define a cosine family generated by $\asoc$, there is but one choice for extensions $\wt{f_i}, i \in \mc K$. Hence, we introduce the extension operator 
\[E\colon   \F \ni\cfk \mapsto \wcfk \in \G, \] where, for all $\cfk\in \F$ (not just for $\cfk \in \dom{\asoc})$,   $\wt{f_i}, i \in \mc K$ are given by \eqref{coo:9} and \eqref{coo:11}. In terms of $E$, \eqref{coo:7} can be written as 
\begin{equation}\label{coo:16} \Cos_{\asoc} (t) = R C_D(t) E, \qquad t\in \R, \end{equation}
and, since,  by the lemma,  $\|E\|\le M$, and clearly $\|R\|\le 1$, we conclude that  $\|\Cos_{\asoc} (t)\| \le M, t \in \R$. 
 
\begin{thm}\label{coo:thm2}
Formula \eqref{coo:16} defines a strongly continuous cosine family in $\F$. Moreover, this cosine family is generated by  $\asoc$.
\end{thm}

\begin{proof}
Let $s\in\R$ and $\cfk\in D(\asoc)$. By Lemma \ref{coo:lem1}, $R C_D(t) E \cfk \in D(\asoc)$ for all $t \in \R$, and the cosine equation for  $\{C_D(t),t\in \R\}$ implies that
\[ R {C}_D(t){C}_D(s)E\cfk=\pol R C_D(t+s)E\cfk+\pol RC_D(t-s) E \cfk \]
belongs to $D(\asoc)$ for $t\in\R$.
By uniqueness of extensions for elements  of $D(\asoc)$, established in Lemma \ref{coo:lem1}, it follows that $C_D(s)E\cfk $ coincides with $E R C_D(s)E \cfk$. Hence, for all $t \in \R$ 
\begin{align*}
2 \Cos_{\asoc}(t)\Cos_{\asoc}(s) \cfk & = 2R{C}_D(t)[E R{C}_D(s)E]\cfk \\
&\hspace{-0.1cm}= 2R{C}_D(t){C}_D(s)E\cfk\\
&\hspace{-0.1cm}=R {C}_D(t+s)E\cfk+R{C}_D(t-s) E \cfk \\
&\hspace{-0.1cm}=\Cos_{\asoc}(t+s)\cfk+\Cos_{\asoc}(t-s)\cfk.
\end{align*}
Since $D(\asoc)$ is dense in $\F$, this proves that $\{\Cos_{\asoc}(t),t\in \R\}$ is a cosine family. The family is strongly continuous since so is the Cartesian product cosine family $\{C_D(t),t\in \R\}$.

Turning to the claim concerning the generator: Let $\cfk\in D(\asoc)$. By Lemma \ref{coo:lem1} each $\wt{f_i}$ belongs to the domain of the basic cosine family and thus we have 
\(\lim_{t \to 0} \tfrac{2}{t^2}[C(t)\wt{f_i} -\wt{f_i}]=\wt{f_i}'', i \in \mc K \)
in the sense of the supremum norm of $\cer$. 
It follows that \(\lim_{t \to 0} \tfrac{2}{t^2}[C_D(t)\wcfk -\wcfk ]=(\wt{f_i}'')_{i\in \mc K}\) in the norm of $\G$. 
Hence, 
$\cfk$ belongs to the domain of the generator, say, $G$, of the cosine family $\{\Cos_{\asoc}(t),t\in\R\}$, and 
$G\cfk=R (\wt{f_i}'')_{i\in \mc K} = (f_i'')_{i\in \mc K}= \asoc \cfk$. On the other hand, $D(\asoc)$ cannot be a proper subset of the domain of $G$ since  $\asoc$ is a Feller generator and, in particular, for $\lambda>0$, $\lambda-\asoc$ is injective and its range is $\F$  (see e.g. \cite{kniga} p.~267).\end{proof}

\begin{rem}A close inspection of the proof of Theorem \ref{coo:thm1} reveals that the space of extensions, that is, the range of $E$, is an invariant subspace for $\{C_D(t), t\in \R\}.$  As restricted to this range,  $\{C_D(t), t\in \R\}$   is a strongly continuous cosine family, and \eqref{coo:16} says that $\{\Cos_{\asoc}, t \in \R\}$ is an isomorphic image of this cosine family. \end{rem}

\section{Convergence of cosine families}\label{sec:coc}
 It may be thought unclear how do we know, in Theorem \ref{coo:thm1}, that $\awsa$  is a cosine family generator. The generation theorem for $\awsa$ can be proved directly, but can also be obtained as a by-product of Proposition \ref{tmt:prop1}, Theorem \ref{conv} and estimate \eqref{coo:2}. Indeed, as proved in \cite{knigazcup}*{Chapter 60}, if  \eqref{coo:2} holds,
Proposition \ref{tmt:prop1} implies existence of a strongly continuous cosine family in $C(\kak)$ given by $\Cos (t) f = \grae \Cos_{\asoce} (t)f , t \in \R, f \in C(\kak)$.  Denoting by $G$ the generator of this family we see also, by the Weierstrass formula, that $\grae \e^{t\asoce } f= \e^{tG}f, t \ge 0, f \in C(\kak)$. Hence, by Theorem \ref{conv}, we need to have  $G=\awsa$, showing that $\awsa$ is a cosine family generator.  

In this section, we want to argue that the abstract Kelvin formula \eqref{coo:16} provides an additional insight into both the convergence theorem (Theorem \ref{coo:thm1}) and the generation theorem for $\awsa$. Details are given in Theorem \ref{coc:thm1}, below. Its notations involve the projection operator $\Pi:\F \to \Fze$  given by 
\[ \Pi \cfk = \Bigl (\sum_{j\in \mc K} \alpha_j f_j\Bigr )_{i \in \mc K}\]
where $\alpha_i, i \in \mc K$ are defined in \eqref{coo:3}. We recall that in Step III of the proof of Lemma \ref{coo:lem1} we used $\Pi$ to denote the operator $\Pi:\R^\ka \to \R^\ka$ which formally acts in the same way as the one introduced here; below, notationally we will not distinguish between these two operators either. 

\begin{thm} \label{coc:thm1}   \   
\begin{itemize}
\item [(a) ] For $f\in C(\kak)$, $\grae \Cos_{\asoce} (t) f = \Cos_{\awsa} (t)f$ uniformly in $t\in \R$.    
\item [(b) ] For $f \in C(\kak)$, $\grae \e^{t\asoce} f =\e^{t\awsa}$  uniformly in $t\ge 0$.    
\item [(c) ] The cosine family generated by $\awsa$ is given by the abstract Kelvin formula 
\begin{equation}\label{coc:1} \Cos_{\awsa} (t) = R C_D (t) E, \qquad t \in \R, \end{equation}
where $E: C(\kak)(\overset{\text{iz}}= \Fze)\to \G $ maps a vector $\cfk\in \Fze$ to the vector $\wcfk\in \G$ of extensions of $f_i$s given by $\wt{f_i}(-x)=g_i(x), x \ge 0, i \in \mc K$ and 
\begin{equation}\label{coc:2} (g_i)_{i\in \mc K} = 2 \Pi \cfk - \cfk .\end{equation}
 \end{itemize} 
\end{thm}
\begin{proof} Condition (b) is a direct consequence of (a) by the Weierstrass formula (see \cite{aberek} for details, if necessary). Also, by Theorem \ref{coo:thm2},  $\Cos_{\asoce} (t)$  is given by \eqref{coo:16} with $E=E(\eps)$ defined in \eqref{coo:9} and \eqref{coo:11} with $Q$ replaced by $\eps^{-1}Q$. 
Therefore, since we already know that the limit cosine family is generated by $\awsa$, to show (a) and (c) simultaneously it suffices to prove that, for any $\cfk$, 
\[ I_\eps (t) \coloneqq \eps^{-1} \int_0^t\big[\e^{s\eps^{-1}Q}Q\big](f_i(t-s))_{i\in \mc K}\textrm{d}s  \]
converges uniformly in $t\ge 0$ to $\Pi (f_i(t))_{i\in \mc K} - (f_i(t))_{i\in \mc K}$.  In the special case of $f_i (x) = const., x \ge 0, i\in \mc K$, this results is immediate, since then the integral is zero, and so is $\Pi (f_i(t))_{i\in \mc K}- (f_i(t))_{i\in \mc K}$. Therefore, we are left with proving this convergence for $\cfk$ such that $f_i(0)=0, i \in \mc K$. 

Under this assumption, given $\epsilon >0$ one can find a $\delta>0$ such that $|f_i(s)|< \epsilon, i \in \mc K$ as long as $s< \delta$. Therefore, by \eqref{coo:14}, for $t\le \delta$, $\|I_\eps (t) \|\le \frac {\epsilon  c M_0 }\eps  \int_0^t\e^{-\eps ^{-1}c\omega  s}\ud s  <  \frac {\epsilon M_0 }\omega$, and so \begin{equation}\label{coc:3}\|I_\eps (t) - \Pi (f_i(t))_{i\in \mc K} + (f_i(t))_{i\in \mc K}  \|\le (2 + \tfrac {M_0 }\omega) \epsilon, \qquad t \in [0,\delta], \eps >0. \end{equation}
Also, there is a $\delta_1$, and without loss of generality we can assume that $\delta_1 <\delta$, such that $|s|< \delta_1$ implies $|f_i(t-s) -f_i(t)|< \epsilon, i \in \mc K, t > \delta$, because $f_i, i \in \mc K$, being members of $\cerp$, are uniformly continuous. Hence, introducing $J_\eps (\delta_1,t) \coloneqq \eps^{-1} \int_0^{\delta_1}\big[\e^{s\eps^{-1}Q}Q\big](f_i(t))_{i\in \mc K} \ud s$, and arguing as above  we obtain
\(\|I_\eps (\delta_1) - J_\eps (\delta_1,t)\| \le \frac {\epsilon M_0}\omega .\)  
 At the same time, using estimate \eqref{coo:14} again, 
\( \|I_\eps (t) - I_\eps (\delta_1)\| \le M_0 \|\cfk\|_\F \e^{-\omega \eps^{-1} \delta_1}. \) Hence, for  $M_1\coloneqq \max (2 + \tfrac {M_0 }\omega, M_0 \|\cfk\|_\F),$ 
\begin{equation} \label{coc:4} \|I_\eps (t) -  J_\eps (\delta_1,t)\|\le M_1 (\epsilon + \e^{-\omega \eps^{-1} \delta_1}), \qquad t \ge \delta, \eps >0. \end{equation}
Finally, we have $J_{\eps}(\delta_1,t)= \e^{\eps^{-1} \delta_1 Q} (f_i(t))_{i\in \mc K}  - (f_i(t))_{i\in \mc K} $ and we know that (a) operators $\e^{\eps^{-1} \delta_1 Q} , \eps >0$ are contractions in $\R^\ka$ (and thus are in particular equibounded) and $\grae \e^{\eps^{-1} \delta_1 Q} v = \Pi v$ for any $v\in \R^\ka$, and (b) the set $K\coloneqq \{u\in \R^\ka ; v = (f_i(t))_{i\in \mc K} \text{ for some } t \ge \delta\}$ is compact in $\R^\ka$. It follows that for sufficiently small $\eps $ 
\[ \sup_{t\ge \delta} \|J_{\eps}(\delta_1,t) - \Pi (f_i(t))_{i\in \mc K}  + (f_i(t))_{i\in \mc K} \| \le \epsilon. \]
This, when combined with \eqref{coc:3} and \eqref{coc:4}, shows that 
\[ \limsup_{\eps \to 0} \sup_{t\ge 0} \|I_{\eps}(t) - \Pi (f_i(t))_{i\in \mc K}  + (f_i(t))_{i\in \mc K} \| \le (M_1 + 1)\epsilon. \]
Since $\epsilon$ is arbitrary, the proof is complete.
\end{proof}

We note that point (c) in the theorem just proved implies that each $\aws$ is a cosine family generator, as long as $\beta=0$. Indeed, given  $\alpha=(\alpha_i)_{i\in \mc K}$ with positive $\alpha_i$ such that $\sum_{i=1}^\ka \alpha_i$ we can take $\mathpzc c \coloneqq (\alpha_i^{-1})_{i\in \mc K}$ and then $\alpha =\alpha (\mathpzc{c}).$

It is also interesting to note that the proof of convergence of the integral $I_\eps (t)$ to $\Pi \cfk - \cfk$ carries out also to $\cfk \in C(S_\ka) \setminus C(\kak)$ except that than the limit function is not continuous at $t=0$, and therefore the limit cannot be uniform. Indeed, only in the first step of the proof, where we look at the case of constant functions, do we see a difference: in the case under consideration the constant depends on the edge, that is, we have $f_i(x)= u_i , x\ge 0, i \in \mc K$ where $u=(u_i)_{i\in \mc K}\in \R^\ka$. Hence, the integral equals $\e^{tQ}u - u$, and this, for $t>0$ converges to $\Pi u - u\not =0$, whereas $\e^{0Q}u - u=0$. Here is an immediate corollary to this remark.

\begin{cor}\label{coc:cor1} For $f \in C(S_\ka) \setminus C(\kak)$, 
the limit \( \grae \Cos_{\asoce} (t) f \) exists for no $t\not =0$.
\end{cor} 

This result, which complements Thm. \ref{coc:thm1} (a), provides a much more specific information than the general theorem of \cite{zwojtkiem}, which says simply that there is at least one $t\not =0$ for which the above limit does not exist.

\bibliographystyle{plain}
\bibliography{../../../bibliografia}

\end{document}